\documentclass[12pt]{amsart}
\usepackage{latexsym, url}
\usepackage{amsthm}
\usepackage{amsmath}
\usepackage{amsfonts}
\usepackage{amssymb}
\usepackage[dvips]{graphicx}
\usepackage{xypic}
\input xy
\xyoption{all}

\newtheorem{theo}{Theorem}[section]
\newtheorem{lemma}[theo]{Lemma}

\newtheorem{rem}[theo]{Remark}
\newtheorem{pb}[theo]{Problem}
\newtheorem{exam}[theo]{Example}

\newcommand\Ind{\operatorname{Ind}}
\newcommand\Inj{\operatorname{Inj}}

\newcommand\op{\operatorname{op}}

\newcommand\Set{\operatorname{\bf Set}}

\newcommand\Lin{\operatorname{\bf Lin}}
\newcommand\FinLin{\operatorname{\bf FinLin}}

\newcommand\colim{\operatorname{colim}}

\newcommand\ca{\mathcal {A}}

\newcommand\cc{\mathcal {C}}
\newcommand\cd{\mathcal {D}}

\newcommand\ck{\mathcal {K}}
\newcommand\cl{\mathcal {L}}

\newcommand\cw{\mathcal {W}}

\date{July 17, 2020}

\begin{document}
\title[Minimal accessible categories]
{Minimal accessible categories}
\author[J. Rosick\'{y}]
{J. Rosick\'{y}}
\thanks{Supported by the Grant Agency of the Czech Republic under the grant 
               19-00902S} 
\address{ 
\newline J. Rosick\'{y}\newline
Department of Mathematics and Statistics\newline
Masaryk University, Faculty of Sciences\newline
Kotl\'{a}\v{r}sk\'{a} 2, 611 37 Brno, Czech Republic\newline
rosicky@math.muni.cz
}
 
\begin{abstract}
We give a purely category-theoretic proof of the result of Makkai and Par\'e saying that the category $\Lin$ of linearly ordered
sets and order preserving injective mappings is a minimal finitely accessible category. We also discuss the existence of a minimal $\aleph_1$-accessible category.
\end{abstract} 

\maketitle

\section{Introduction}
One of striking results of \cite{MP} is that the category $\Lin$ of linearly ordered sets and order preserving injective mappings
is a minimal finitely accessible category. This means that for every large finitely accessible category $\ck$ there is a faithful functor $\Lin\to\ck$ preserving directed colimits. \cite{MP} does not contain a proof of this result -- Makkai and Par\'e just
say that it essentially follows from the work of Morley \cite{M}. Since there are many applications of this result (see, e.g., \cite{LR}), it might be useful to give an explicit proof of it. We do it by transferring the standard model-theoretic argument to the language of accessible categories. Another, more model-theoretic proof, of the theorem of Makkai and Par\'e was recently given by Boney \cite{Bo}.

The minimality of $\Lin$ among finitely accessible categories implies its minimality among ($\infty,\omega$)-elementary categories (see \cite{MP} 3.4.1) and, even, among accessible categories
with directed colimits whose morphisms are monomorphisms (\cite{LR} 2.5).
One cannot expect that $\Lin$ is a minimal accessible category because there is no faithful functor from $\Lin$ to the 
$\aleph_1$-accessible category of well ordered sets and order preserving injective mappings. The reason is that any well ordered set $A$ is iso-rigid, it means that every isomorphism $A\to A$ is the identity. Using \cite{HPT}, we give an example 
of a $\aleph_1$-accessible category $\ck$ having every object $K$ rigid, i.e., every morphism $K\to K$ is the identity.
This yields a candidate for a minimal $\aleph_1$-accessible category. Similarly, one gets a candidate for a minimal 
$\aleph_\alpha$-accesible category.

\vskip 1mm
\noindent
{\bf Acknowledgement.}  We are grateful to T. Beke for useful discussions concerning this paper.

\section{Skolem cover}
Let $\ck$ be a finitely accessible category and $\ca$ its representative small full subcategory of finitely presentable objects
(i.e., any finitely presentable object of $\ck$ is  isomorphic to some $A\in \ca$). Let 
$$
E:\ck\to\Set^{\ca^{\op}}
$$ 
be the canonical embedding that takes each $K\in\ck$ to the contravariant functor $\ck(-,K):\ca\to\Set$. We note that, by Proposition 2.8 in \cite{AR},
this functor is fully faithful and preserves directed colimits and finitely presentable objects. Following Theorem 4.17 in \cite{AR}, $\ck$
is equivalent to a finitary-cone-injectivity class $\Inj(T)$ in $\Set^{\ca^{\op}}$; this means that there is a set $T$ of cones $a=(a_i:X\to EA_i)_{i\in I}$ 
where $X$ is finitely presentable in $\Set^{\ca^{\op}}$ and $A_i\in \ca$, $i\in I$ such that $\Inj(T)$ consists of functors $F$ injective to each cone 
$a\in T$. The latter means that for any morphism $f:X\to F$ there is $i\in I$ and $g:EA_i\to F$ with $ga_i=f$.
Let $S(\ck)$ be the category whose objects are $(F,a_F)_{a\in T}$ consisting of $F:\ca^{\op}\to\Set$ with $a_F$ assigning to a cone $a$ and
$f:X\to F$ a morphism $a_F(f):EA_i\to F$ for some $i\in I$ such that $a_F(f)a_i=f$. Morphisms $(F,a_F)\to (F',a_{F'})$ are
natural transformations $\varphi:F\to F'$ such that $a_{F'}(\varphi f)=\varphi a_F(f)$. The forgetful functor $G:S(\ck)\to\Set^{\ca^{\op}}$
is faithful and has values in $\Inj(T)$. Its codomain restriction $S(\ck)\to\Inj(T)$ is surjective on objects. Since $E:\ck\to\Inj(T)$ is an equivalence,
we get a faithful functor $H:S(\ck)\to\ck$ which is essentially surjective on objects, i.e., any $K\in\ck$ is isomorphic to some $H(F,\tilde{a})$.

\begin{lemma}\label{le2.1}
The category $S(\ck)$ is finitely accessible and $H:S(\ck)\to\ck$ preserves directed colimits.
\end{lemma}
\begin{proof}
Let $D:\cd\to S(\ck)$ be a directed diagram and consider the colimit $\delta: GD\to F$ in $\Set^{\ca^{\op}}$. Then $\colim D =(F,a_F)$ where 
$a_F(f)=\delta_d a_{Dd}(g)$ where $f=\delta_d g$. Since $X$ is finitely presentable, the description is correct. Thus $S(\ck)$ has directed
colimits and $G$ preserves them. Hence $H$ preserves them too.

If $F$ is finitely presentable in $\Set^{\ca^{\op}}$ then any $(F,a_F)$ is finitely presentable in $S(\ck)$. In order to show that any $(F,a_F)$
is a directed colimit of finitely presentable objects in $S(\ck)$ it suffices to express $F$ as a directed colimit of finitely presentable objects
$F_d$ in $\Set^{\ca^{\op}}$ and complete them to $(F_d,a_{F_d})$ using finite presentability of $X$ again. Then $(F,a_F)$ is a directed colimit of
$(F_d,a_{F_d})$. Thus $S(\ck)$ is finitely accessible.
\end{proof}  

In fact, we have shown that
$$
S(\ck)=S(\Ind\ca)=\Ind S(\ca)
$$
$S(\ck)$ will be called a \textit{Skolem cover} of $\ck$ because it is a skolemization of the $L_{\infty,\omega}$-theory corresponding to $T$.

Let $U:\Set^{\ca^{\op}}\to\Set$ assign to $F$ the set $\coprod_{A\in\ca}FA$. The functor $U$ is faithful and preserves directed colimits. Thus
$(\ck,UE)$ and $(S(\ck),UG)$ are concrete finitely accessible categories with concrete directed colimits and $H:S(\ck)\to \ck$ is a concrete
functor. 

\begin{lemma}\label{le2.2}
Let $(F,a_F)\in S(\ck)$ and $Z\subseteq UG(F,a_F)$. Then there is the smallest subobject $(F_Z,a_{F_Z})$ of $(F,a_F)$ such that $Z\subseteq UGF_Z$.
\end{lemma}
\begin{proof}
Let $F_0$ be the smallest subfunctor of $F$ such that $Z\subseteq UF_0$; let $\sigma:F_0\to F$ denote the inclusion. Consider a cone $a:X\to EA_i$ 
in $T$ and a morphism $f:X\to F_0$. Then the composition $\sigma f$ factorizes through some $a_i$. Let $F_1$ be a colimit in $\Set^{\ca^{\op}}$
of the diagram 
$$
\xymatrix{
F_0 \\
&&\\
X\ar[uu]^f  \ar@{.}[ur] \ar [rr]_{a_i} && EA_i
}
$$
consisting of all spans $(f,a_i)$ above. We iterate this construction by replacing $F_0$ with $F_1$, etc. In this way, we get the chain
$F_0\to F_1\to\dots F_n\to\dots$. Then $F_Z=\colim F_n$ carries the desired smallest subobject of $(F,a_F)$.
\end{proof}

This is the virtue of the skolemization and reflects the fact that the skolemized theory is universal. We skolemized cone-injectivity while
algebraic factorization systems (see \cite{GT}) skolemize injectivity. J. Bourke \cite{B} came to the same point from a different motivation.

\begin{rem}\label{re2.3}
{
\em
For any $Z$, there is only a set of non-isomorphic $(F_Z,a_{F_Z})$, $F:\ca^{\op}\to\Set$.
}
\end{rem}

\section{Minimal finitely accessible categories}
\begin{theo}\label{th3.1}
For any large finitely accessible category $\ck$ there is a faithful functor $\Lin\to\ck$ preserving directed colimits.
\end{theo}
\begin{proof}
Following \ref{le2.2}, we can assume that $\ck$ is equipped with a faithful functor $U:\ck\to\Set$ preserving directed colimits and such that
for any subset $Z\subseteq UK$ there is the smallest subobject $K_Z$ of $K$ such that $Z\subseteq UK_Z$. Let $\cl$ be the category with objects
$(K,X)$ where $K\in\ck$ and $X\subseteq UK$ is linearly ordered. Morphisms $(K_1,X_1)\to (K_2,X_2)$ are morphisms $f:K_1\to K_2$ such that
$Uf$ induces the order preserving mapping $X_1\to X_2$. The category $\cl$ has directed colimits given as
$$
\colim (K_i,X_i)=(\colim K_i,\colim X_i)
$$
and any $(K,X)$ with $K$ finitely presentable in $\ck$ and $X$ finite is finitely presentable in $\cl$. Thus $\cl$ is finitely accessible
and the forgetful functor $\cl\to\ck$ preserves directed colimits.

For a $\cl$-object $(K,X)$, let $\rho_{(K,X)}$ be the greatest ordinal $0<\rho_{(K,X)}\leq\omega,|X|$ such that for any $n<\rho_{(K,X)}$
and any $a_1< a_2<\dots <a_n$ and $b_1< b_2<\dots < b_n$ in $X$ there is an isomorphism $s:K_{\{a_1,\dots,a_n\}}\to K_{\{b_1,\dots,b_n\}}$ such
that $Us(a_i)=b_i$ for $i=1,\dots,n$. Assume that there is $(K,X)\in\cl$ with $\rho_{(K,X)}=\omega$. Then $X$ is infinite and for any $n<\omega$
there is a chain $a_{n1}< a_{n2}<\dots < a_{nn}$ in $X$. We will construct a functor $F:\Lin\to\ck$ as follows. Finitely presentable objects in $\Lin$
are finite chains $C_n$ with elements $1< 2<\dots <n$. Put $F_0(C_n)=K_{a_{n1},\dots,a_{nn}}$. Given an injective order preserving mapping $h:C_m\to C_n$, let
$Fh$ be the composition
$$
K_{a_{m1},\dots,a_{mm}}\to K_{a_{nh(1)},\dots,a_{nh(m)}}\to K_{a_{n1},\dots,a_{nn}}
$$
where the first morphism is the isomorphism $s$ above and the second morphism is the inclusion. Given $h_1:C_k\to C_m$ and $h_2:C_m\to C_n$ then
it is easy to see that $F_0(h_2h_1)=F_0(h_2)F_0(h_1)$.  In fact, we always get the isomorphism 
$$
K_{a_{k1},\dots,a_{kk}}\to K_{a_{nh_2h_1(1)},\dots a_{nh_2h_1(k)}}
$$
followed by the inclusion $K_{a_{nh_2h_1(1)},\dots a_{nh_2h_1(k)}}\to K_{a_{n1},\dots,a_{nn}}$. Thus we get the functor $F_0:\FinLin\to\ck$
defined on finite linear orderings. Since $\Lin=\Ind\FinLin$, $F_0$ extends to a functor $F:\Lin\to\ck$ preserving directed colimits.
Since $F_0$ is faithful, $F$ is faithful too.

Assume that $\rho_{(K,X)}<\omega$ for any $(K,X)\in\cl$. We put $(K_1,X_1)<(K_2,X_2)$ provided that $\rho_{(K_2,X_2)}<\rho_{(K_1,X_1)}$ and
$(K_1)_{\{a_1,\dots,a_{\rho_{(K_2,X_2)}}\}}\cong (K_2)_{\{b_1,\dots b_{\rho_{(K_2,X_2)}}\}}$ for any $a_1< \dots < a_{\rho_{(K_2,X_2)}}$ in $X_1$
and any $b_1,\dots <b_{\rho_{(K_2,X_2)}}$ in $X_2$. Then $<$ partially orders objects of $\cl$ and this order is well-founded in the sense
that there is no decreasing chain
$$
\dots <(K_n,X_n) < (K_{n-1},X_{n-1}) <\dots < (K_1,X_1).
$$
Such chain would yield a diagram
$$
(K_1)_{\{a_{11}\}}\to (K_2)_{\{a_{21},a_{22}\}}\to (K_n)_{\{a_{n1},\dots,a_{nn}\}}
$$
whose colimit $(K,X)$ in $\cl$ has $\rho_{(K,X)}=\omega$. Thus we can assign an ordinal $\alpha(K,X)$ to each $(K,X)\in\cl$ in such a way
that 
$$
\alpha(K,X)=\sup_{(K',X')<(K,X)}\alpha(K',X')+1.
$$

Following \ref{re2.3}, there is an infinite cardinal $\mu$ greater or equal to the number of non-isomorphic objects $K_X$ for $X$ finite
and $K$ arbitrary. For $(K,X)\in\cl$, choose $a_1<\dots <a_{\rho_{(K,X)-1}}$ in $X$ and put
$$
(K,X)^\ast= (K_{\{a_1,\dots,a_{\rho_{(K,X)-1}}\}},X\cap UK_{\{a_1,\dots,a_{\rho_{(K,X)-1}}\}}).
$$
We will prove that 
$$
|X|<\exp_{\omega(\alpha(K,X)^\ast+1})(\mu) 
$$
for any $(K,X)\in\cl$. Recall that $\exp_0(\mu)=\mu$, $\exp_{\xi+1}(\mu)=2^{\exp_\xi(\mu)}$ and $\exp_\eta(\mu)=sup_{\xi<\eta}\exp_\xi(\mu)$.
Since $(K,UK)\in\cl$ for any $K$ in $\ck$, this inequality implies that $\ck$ is small.

The proof will use the recursion on $\alpha(K,X)^\ast$.
Let $\alpha(K,X)^\ast=0$ and assume that $|X|\geq\exp_\omega(\mu)$.  The set $X^n$ is decomposed into $\leq\mu$ parts following isomorphisms types
of $K_{\{a_1,\dots,a_n\}}$. Following the Erd\"os-Rado partition theorem (see \cite{J}, Exercise 29.1), there is $X_0\subseteq X$ such that $X_0>\mu$
and $K_{\{a_1,\dots,a_n\}}\cong K_{\{b_1,\dots,b_n\}}$ for any $a_1<\dots<a_n$ and $b_1<\dots<b_n$ in $X_0$. Thus 
$$\label{claim}
(K,X_0)^\ast=(K_{\{a_1,\dots,a_n\}},X_0\cap UK_{\{a_1,\dots,a_n\}})<(K,X)^\ast,
$$
which is impossible because $\alpha(K,X)^\ast=0$.

Assume that the claim holds for any $(K,X)\in\cl$ with $\alpha(K,X)^\ast<\beta$ and consider $(L,Y)\in\cl$ with $\alpha(L,Y)^\ast=\beta$. Assume that
$|Y|\geq\exp_{\omega(\alpha(L,Y)^\ast+1)}(\mu)$ and let $n=\rho_{(L,Y)}$. We have
$$
|Y|\geq\exp_{\omega(\beta+1)}(\mu)>\exp_{\omega\beta+n-1}(\mu)=\exp_{n-1}(\exp_{\omega\beta}(\mu).
$$
Following the Erd\"os-Rado partition theorem, there is $Y_0\subseteq Y$ such that $|Y_0|>\exp_{\omega\beta}(\mu)$ and 
$L_{\{b_1,\dots,b_n\}}\cong L_{\{c_1,\dots,c_n\}}$ for each $b_1<\dots <b_n$ and $c_1<\dots <c_n$ in $Y_0$. Then $\rho_{(L,Y_0)}> n$ and
$(L,Y_0)<(L,Y)$. Thus $(L,Y_0)^\ast < (L,Y)^\ast$. Hence $\alpha(L,Y_0)^\ast<\beta$ and thus
$$
|Y_0|<\exp_{\omega(\alpha(L,Y_0)^\ast+1}(\mu)\leq\exp_{\omega\beta}(\mu),
$$
which is a contradiction.
\end{proof}

\section{Towards minimal $\lambda$-accessible categories}

\begin{exam}\label{re4,2}
{
\em
The category $\cw$ of well-ordered sets is $\aleph_1$-accessible and any its object $K$ is iso-rigid in the sense that the only isomorphism
$K\to K$ is the identity. Thus there is no faithful functor $\Lin\to\cw$ and a prospective minimal $\aleph_1$-accessible category is iso-rigid.
}
\end{exam}

\begin{exam}\label{hpt}
{
\em
There is an $\aleph_1$-accessible category $\cl$ having all objects $K$ rigid in the sense that the only morphism $K\to K$ is the identity. Thus there is no faithful functor $\cw\to\cl$.

The construction of $\cl$ is motivated by \cite{HPT}, II.3. Let $\ck$ be the category of structures $(A,<,R,S,\sup,s)$ where $<$ is a well-ordering, 
$R$ is a unary relation, $S$ is an $\omega$-ary relation, $\sup$ is the countable join and $s$ is the unary operation of taking the successor. 
Let $T$ be the following set of axioms:
\begin{enumerate}
\item $(\forall x_0,x_1,y_1,\dots,x_n,y_n,\dots)(S(x_0,x_1,\dots,x_n,\dots)\wedge S(x_0,y_1,\dots,y_n,\dots)\to \newline\bigwedge_{0<n}x_n=y_n)$

\item $(\forall x_0,x_1,\dots,x_n,\dots)(S(x_0,x_1\dots,x_n\dots)\to(\bigwedge_{0<n}x_n<x_{n+1})\wedge x_0=\sup x_n)$

\item $(\forall x)(\exists(y_1,\dots,y_n,\dots)(\bigwedge_{0<n}(y_n<y_{n+1})\wedge x=\sup y_n)\newline \to(\exists x_1,\dots,x_n,\dots)S(x,x_1,\dots,x_n,\dots)$
       
\item $(\forall x)(R(x)\leftrightarrow \neg(\exists y)(x=s(y))\wedge \neg(\exists x_1,\dots,x_n,\dots)S(x,x_1,\dots,x_n,\dots)$
\end{enumerate}
Let $A_2$ be the set of isolated elements of $A$, $A_0$ be the set of all limit elements of $a\in A$ such that $S(a,a_1,\dots,a_n,\dots)$ for some
$a_1,\dots,a_n,\dots\in A$ and $A_1=A\setminus (A_0\cup A_2)$. All the sets $A_2$, $A_0$ and $A_1$ are preserved by homomorphisms $f:A\to B$
(due to $s$, $S$ and $R$ resp.).  

This category clearly has $\aleph_1$-directed colimits. Objects $A$ of $\cl$ generated by $0$ are ordinals $\omega_1$ with a choice of $S$ for every $a\in A_0$. Thus there is $\aleph_0^{\aleph_1}=2^{\aleph_1}$ such objects. These objects are
$\aleph_1$-presentable and the same is true for objects $\omega_1\cdot\alpha$ where $\alpha<\omega_1$. Clearly, every object
of $\cl$ is an $\aleph_1$-directed colimit of these objects $\omega_1\cdot\alpha$, $\alpha<\omega_1$. Thus $\cl$ is 
$\aleph_1$-accessible.

Assume that there exists a morphism $f:A\to A$ in $\cl$ which is not the identity. Let $a$ be the least element in $A$ such that $f(a)\neq a$.
Since A is a well-ordered set and $f$ is injective, $a<f(a)$. Hence
$$
a<f(a)<f^2(a)<\dots <f^n(a)<\dots
$$
Let $b=\sup f^n(a)$. There are $b_1<b_2<\dots <b_n<\dots$ such that $S(b,b_1,\dots,b_n,\dots)$. Since $S(f(b),f(b_1),\dots,f(b_n),\dots)$,
we have $f(b)=\sup f(b_n)$. For each $n$ there is $k$ such that $b_n<f^k(a)$. Hence $f(b_n)<f^{k+1}(a)$ and thus $f(b)=b$. Therefore $f(b_n)=b_n$
for each $n$. Since $a<b_m$ for some $m$, $f^n(a)<b_m$ for each $n$. Hence $b\leq b_m$, which is a contradiction.
}
\end{exam}

\begin{rem}
{
\em
(1) Let $\cl_1$ be a full subcategory of $\cl$ where we choose $S$ for every $a\in A_0$ in every object generated by $0$.  This category does not depend of the choices of $S$ and is also $\aleph_1$-accessible. In fact, it is $\Ind_{\aleph_1}(\cc_1)$ where
$\cc_1$ is the category of ordinals $\omega_1\cdot\alpha$, $\alpha<\omega_1$ with non-identity morphisms 
$$
\omega_1\cdot f:\omega\cdot\alpha\to\omega_1\cdot\beta
$$
where $f:\alpha\to\beta$ is an order preserving injective mapping with $\alpha<\beta$. The category $\cc_1$ is, in fact,
the category of ordinals $\alpha<\omega_1$ where non-identity morphisms are order preserving injective mappings $\alpha\to\beta$
for $\alpha<\beta<\omega_1$.

(2) The category $\FinLin$ is the category $\cc_0$ of ordinals $\alpha<\omega$ where non-identity morphisms are order preserving
injective mapping $\alpha\to\beta$ for $\alpha<\beta<\omega$. Observe that $\FinLin$ is rigid, i.e., the only morphisms
$\alpha\to\alpha$ are the identities. Hence $\cl_1=\Ind_{\omega_1} \cc_1$ is $\aleph_1$-modification of a minimal $\aleph_0$-accessible category $\Lin$.

(3) Let $\cc_\gamma$ be the category of ordinals $\alpha<\omega_\gamma$ where non-identity morphisms are order preserving
injective mapping $\alpha\to\beta$ for $\alpha<\beta<\omega_\gamma$. Then $\cl_\gamma = \Ind_{\aleph_\gamma}\cc_\gamma$ is
an $\aleph_\gamma$-accessible category.
}
\end{rem}

\begin{pb}
{
\em
Is $\cl_1$ a minimal $\aleph_1$-accessible category? This means that for every large $\aleph_1$-accessible category $\ck$ there is a faithful functor $\cl_1\to\ck$ preserving $\aleph_1$-directed colimits.

Similarly, is $\cl_\gamma$ a minimal $\aleph_\gamma$-accessible category
for $0<\gamma$?
}
\end{pb}

\end{document}